\documentclass[12pt]{amsart}

\usepackage{amssymb}

\usepackage{enumitem}
\usepackage{mathrsfs}
\usepackage{srcltx}
\usepackage[bookmarks]{hyperref}

\makeatletter
\@namedef{subjclassname@2020}{%
  \textup{2020} Mathematics Subject Classification}
\makeatother

\usepackage[T1]{fontenc}

\newtheorem{thm}{Theorem}[section]
\newtheorem{cor}[thm]{Corollary}

\theoremstyle{definition}
\newtheorem{dfn}[thm]{Definition}

\numberwithin{equation}{section}

\newcommand{\R}{\mathbb{R}}
\newcommand{\Cc}{\mathscr{C}}
\newcommand{\D}{\mathcal{D}}
\newcommand{\G}{\mathbb{G}}
\newcommand{\spl}{\mathcal{D}^+}
\newcommand{\om}{\infty}
\newcommand{\ep}{\varepsilon}
\newcommand{\et}{\quad \mbox{and} \quad }

\def\adh#1{\overline{#1}}

\begin{document}


\baselineskip=17pt


\title[]{A remark on $\Cc^\infty$ definable equivalence}

\author[A. Valette]{Anna Valette}
\address{Chair of Optimization and Control\\ Jagiellonian University\\
{\L}oja\-sie\-wi\-cza 6\\
30-348 Krak\'ow, Poland}
\email{anna.valette@im.uj.edu.pl}

\author[G. Valette]{Guillaume Valette}
\address{Institute of Mathematics\\ Jagiellonian University\\
{\L}ojasiewicza 6\\
30-348 Krak\'ow, Poland}
\email{guillaume.valette@im.uj.edu.pl}

\date{}

\begin{abstract}
We establish that if a submanifold $M$ of $\R^n$ is definable in some o-minimal structure then 
any definable submanifold $N\subset \R^n$ which is $\Cc^\infty$ diffeomorphic to $M$, with a diffeomorphism $h:N\to M$ that is sufficiently close to the identity,  must be $\Cc^\infty$ definably diffeomorphic to $M$.  The definable diffeomorphism  between $N$ and $M$ is then provided by a tubular neighborhood of $M$.
\end{abstract}



\maketitle

\section{Introduction}

The framework of o-minimal structures is well adapted to perform analysis as well as
 geometry, and definable mappings have many finiteness properties that are valuable for applications. It  is however not always easy to construct  definable diffeomorphisms. For instance, smooth trivializations are often generated by integration of a vector field \cite{ehresmann} and the flow of a definable vector field may fail to be definable in the same structure.   
Furthermore, M. Shiota gave examples of algebraic smooth manifolds that are $\Cc^\infty$ diffeomorphic
but not Nash diffeomorphic (i.e. not semialgebraically $\Cc^\infty$ equivalent) \cite{shiota}. This points out that   $\Cc^\infty$ equivalence of definable  manifolds does not yield definable $\Cc^\infty$ equivalence. In this note,
we show that such pathologies cannot arise if the  diffeomorphisms considered are sufficiently close to the identity. Namely,
 we establish that if a submanifold $M$ of $\R^n$ is definable in some o-minimal structure then 
any definable submanifold $N\subset \R^n$ which is $\Cc^\infty$ diffeomorphic to $M$, with a diffeomorphism $h:N\to M$ that is sufficiently close to the identity,  must be definably $\Cc^\infty$  diffeomorphic to $M$ (Theorem \ref{iso}).  The definable diffeomorphism  between $N$ and $M$ is then provided by a tubular neighborhood of $M$.

We briefly recall that an \textbf{o-minimal structure} expanding the real  field $(\R, +,\cdot)$ is the data for every $n$ of  a Boolean algebra 
 $\mathcal{D}_n$  of subsets of $\R^n$   containing  all the algebraic subsets of $\R^n$ and satisfying
the following axioms:
\begin{enumerate}
\item If $A\in\mathcal{D}_m$, $B\in\mathcal{D}_n$, then $A\times B\in\mathcal{D}_{m+n}$; \item
If $\pi:\R^n\times \R\to \R^n$ is the natural projection and $A\in\mathcal{D}_{n+1}$, then $\pi(A)\in\mathcal{D}_n$;
 \item $\mathcal{D}_1$ is nothing else but all the finite unions of points and intervals.
\end{enumerate}
A set belonging to the structure $\mathcal{D}$ is called a {\it definable set} and a map whose graph is in the structure $\mathcal{D}$ is called  a {\it definable map}. Given $B\subset \R^k$, we say that $(Z_t)_{t\in B}$ is  a {\it definable family} of subsets of $\R^n$ if for each $t\in B$, $Z_t\subset \R^n$  and  $\bigcup_{t\in B}\{t\} \times Z_t\in  \mathcal{D}_{k+n}$ (in particular, $Z_t$ is  definable  for all $t$). A family of mappings $(\varphi_t)_{t\in B}$ is said to be definable if the family of the respective graphs of the mappings $\varphi_t$ is a definable family of sets.

 Given  a definable set $A\subset\R^n$, we denote by  $\spl(A)$ the set of positive definable continuous functions on $A$.
 
 Given a mapping $f:A\to B$, with $A\subset \R^n$, $B\subset \R^k$, we denote by $|f|$ the function which assigns to $x\in A$ the number $|f(x)|$. Here we stress the fact that since $|f|$ is not a real number but a function, this does not define a norm on the space of mappings.

If $f$ is a differentiable mapping, we  denote by $d_xf$ its derivative at $x$ and we will write $|d_xf|$ for the norm of $d_xf$ (as a linear mapping) derived from the euclidean norm $|.|$.   We will write $|df|$ for the function defined by $x\mapsto |d_xf|$. We then set $$|f|_1:=|f|+|df|.$$  

 We will denote by $B(x,r)$ the open euclidean ball of radius $r$ centered at $x$, and by $\adh{A}$ the closure of $A$ in the euclidean topology.
 
 \section{Definable diffeomorphisms via retractions}
 Let us recall that 
 given a definable $\Cc^p$ submanifold $M$ of $\R^n$, $p\ge 2$, there is a definable  neighborhood $U$ of $M$ in $\R^n$ and a definable retraction $r:U\to M$ such that for all $x\in U$, $r(x)$ is the point that realizes the distance from $x$ to $M$. The vector $(x-r(x))$ is then orthogonal to the tangent space to $M$ at $r(x)$ and we say that $(U,r)$ is {\it a tubular neighborhood of $M$}. The mapping $r$ is at least $\Cc^{p-1}$ and, if $M$ is 
 $\Cc^\om$, then so is $r$ \cite{polyraby}, \cite[Proposition 2.4.1]{livre}, see also \cite[Theorem 6.11]{coste_bleu}.
\begin{thm}\label{iso}
Let $M\subset\R^n$ be a closed definable $\Cc^2$ submanifold and let $(U,r)$ be a definable tubular neighborhood of $M$.
There  exists $\ep\in\D^+(\R^n)$ such that if $N\subset\R^n$  is any definable $\Cc^2$ submanifold for which there exists a $\Cc^1$ diffeomorphism $h:N\to M$  satisfying $|h-id_{|_N}|_1<\ep$ then $N$ is contained in $U$ and the restriction $r_{|_N}:N\to M$ is a $\Cc^1$ diffeomorphism.
\end{thm}

\begin{proof}
 Let $N\subset\R^n$  be a  definable $\Cc^2$ submanifold and $h:N\to M$ a diffeomorphism such that  $|h(x)-x|_1<\ep(x)$ for some $\ep\in\D^+(\R^n)$. We assume $\ep<1$ and we will put extra requirements on  $\ep$ on the way. For $\delta\in\D^+(M)$ define
 \begin{equation}\label{udelta}U_\delta:=\{x\in U\; : dist(x, M)<\delta(r(x))\},\end{equation}
 where $dist(x,M)$ is the euclidean distance of $x$ to $M$. If  $\delta$ is sufficiently small we have  $\adh{U_\delta}\subset U$.  Replacing $U$ with $U_\delta$ for some small $\delta$, we can assume  that $|d_xr|$ is bounded on bounded sets. Moreover, the assumption  $|h-id_{|_N}|_1<\ep$  entails $dist(x,M)<\ep(x)$ for every $x$ in $N$ and therefore $N\subset U_{\delta}$ if
 $\ep(x)\le\eta(x):=\inf\{\frac{\delta(z)}{2},\, |z|<|x|+1\}$.

Let $x\in N$ and take a unit vector $u\in T_xN$. Put $v:=d_xh(u)$. By assumption, we have $|u-v|<\ep(x)$.
Observe that 
\begin{enumerate}
\item $|d_xr(u)-d_xr(v)|\le |u-v||d_xr|$,
\item $|d_{h(x)}r(v)-u|=|v-u|$ (as $r$ is the identity on $M$),
\item by continuity of $x\mapsto d_xr$, there is a definable function $\mu(x)>0$ such that if $|x-h(x)|<\mu(x)$ then $|d_xr(v)-d_{h(x)}r(v)|<\frac{1}{8}$.
\end{enumerate}
By the above, if $\ep(x)<\min\{\eta(x),\mu(x),\frac{1}{8(|d_xr|+1)}\}$ we have
\begin{eqnarray*}|d_xr(u)-u|&\le& |d_xr(u)-d_xr(v)|+|d_{h(x)}r(v)-u|+|d_xr(v)-d_{h(x)}r(v)|\\
&<&\ep(x)(|d_xr|+1)+\frac{1}{8}<\frac{1}{4},\end{eqnarray*}
which shows that $d_x(r_{|N})$ is an isomorphism, implying that $r$ induces a local diffeomorphism on $N$.

To prove that it is a diffeomorphism on $N$, we start by showing that the restriction $r_{|_N}$ is proper. Observe that $N$ is closed. Indeed, take $(x_i)\subset N$ such that $x_i\to x\in \adh{N}$. As $h(x_i)$ is bounded and $M$ is closed there is  a  $y\in M$  such that $h(x_i)\to y$ (extracting some subsequence if necessary), which implies $x=h^{-1}(y)\in N$, yielding $N$ is closed. Now,   notice that if $|x_i|\to+\infty$ then, as $|r(x_i)-x_i|\le |h(x_i)-x_i|<\ep(x_i)$,  the sequence $r(x_i)$ must go  to  infinity as well. Since $N$ is closed, this shows that  the restriction of $r$ to $N$ is proper. 

Hence,
by Ehresmann's theorem $r_{|_N}$ is a locally trivial fibration above every connected component of $M$. 
We will show that it is indeed one-to-one. 
For this purpose, let us fix any $x$ in $M$. There is $q>0$ such that  $V_x:=B(x,q)\cap M$ is simply connected and hence $r_{|_N}$ is trivial above $V_x$.
We first show by way of contradiction that if $\ep$ is small enough (depending on $x$) then $r^{-1}(x)\cap N$ must be reduced to a single point.
Write 
 $r_{|_N}^{-1}(x)=\{x_1,\dots ,x_k\}$, suppose $k\ge 2$,   and notice that $$r_{|_N}^{-1}(V_x)=V_{1}\cup\dots \cup V_{k}$$ 
 is a disjoint union of neighborhoods $V_{i}$ of $x_i$ in $N$ respectively.
Take now $y\in B(x,q/8)\cap M$ and set $\ep_0:=\sup_K \ep$, where  $K$ is some compact neighborhood of $x$ comprising all the sets that we are considering. Let us observe that
$$|h^{-1}(y)-x|\le |h^{-1}(y)-y|+|y-x|<\ep_0+q/8,$$ which means that for $ \ep_0<q/16$, we have $h^{-1}(y)\in B(x,q/4)$, and consequently $$h^{-1}(B(x,q/8)\cap M)\subset B(x,q/4).$$ 
Moreover, since $$|r(h^{-1}(y))-x|\le |r(h^{-1}(y))-h^{-1}(y)|+|h^{-1}(y)- y|+|y-x|<q/16+q/16+q/8$$ we have that $r(h^{-1}(y))\in B(x,q/4)$, and hence 
$$h^{-1}(B(x,q/8)\cap M)\subset V_{1}\cup\dots \cup V_{k}.$$ Since the last union is disjoint, we can assume that $$h^{-1}(B(x,q/8)\cap M)\subset V_{1}.$$
Furthermore, we have 
$$|h(x_2)-x|<|h(x_2)-x_2|+|x_2-x|<2\ep_0<q/8,$$ 
and hence $x_2\in h^{-1}(B(x,q/8)\cap M)$, which  contradicts the fact that $V_1$ and $V_2$ are disjoint, establishing  that $r^{-1}(x)\cap N$ must be reduced to a single point.

Choose now one point in each  connected component of $M$, say $z_1,\dots,z_l$. By the above, for $\ep$ small enough, the set $r^{-1}(z_i)\cap N$ is reduced to a single point for all $i\le l$. Since $r_{|_N}:N\to M$ is a locally trivial fibration, it must be one-to-one.

We turn to show that $r_{|N}$ is onto.  As it is a locally trivial fibration on each connected component of $M$, it suffices to show that $r(N)$ contains at least one point in every connected component. Take any $x\in M$ and observe that, for $\ep$ small enough, $h^{-1}(x)\in U$ and $r(h(x))$ is a point close to $x$, which therefore must belong to the same connected component of $M$ as $x$, if $\ep$ is sufficiently small.  
\end{proof}
\begin{dfn} Let  $M$ be a $\Cc^k$ submanifold of $\R^n$, $k\ge 2$ (possibly infinite).  Fix any $\ep\in \D^+(M)$ and  a positive integer $p\le k$. A {\it definable deformation of $M$} is a definable family  $(Z_t)_{t\in[0,1]}$  of $\Cc^p$ submanifolds $Z_t\subset \R^n$ with  $M=Z_0$.  A deformation  $(Z_t)_{t\in[0,1]}$ is {\it $(\ep,\Cc^p)$   trivial} if
there exists a family of $\Cc^p$ diffeomorphisms $\varphi_t:M\to Z_t$, $t\in[0,1]$, $\Cc^p$ with respect to $t$ and satisfying $\varphi_0(x)=x$ for all $x$ as well as: 
     \begin{equation}\label{eq_istopie_approx}
      |x-\varphi_t(x)|< \ep(x) \et |u-d_x \varphi_t(u)|<\ep(x),
     \end{equation}
 for every $x\in M=Z_0$ and every unit vector $u\in T_xM$.  When $(\varphi_t)_{t\in [0,1]}$ is a definable family of mappings, we say that $(Z_t)_{t\in[0,1]}$  is {\it definably  $(\ep,\Cc^p)$ trivial}.
\end{dfn}
\begin{cor}
Let $M\subset\R^n$ be a closed, definable $\Cc^k$, $k\ge 2$ (possibly infinite), submanifold and let
$\ep\in \D^+(M)$. There exists $\delta\in\D^+(M)$ such that any $(\delta,\Cc^{k-1})$ trivial definable  deformation of $M$ is  $(\ep,\Cc^{k-1})$ definably trivial.
\end{cor}
\begin{proof}Let 
$\ep\in \D^+(M)$ and take a tubular neighborhood $(U,r)$ of $M$. Take some $\delta$ sufficiently small for  $U$ to contain the closure of $U_\delta$ (see (\ref{udelta})).  The derivative of $r$  is then uniformly continuous on every bounded subset of $U_\delta$. Let   $(Z_t)_{t\in[0,1]}$ be a $(\delta,\Cc^k)$ trivial deformation, with corresponding family of diffeomorphisms $\varphi_t:M\to Z_t$. If $\delta$ is sufficiently small, by Theorem \ref{iso} (applied with $h:=\varphi_t^{-1}$ for each $t$), the restriction $r_t$ of $r$ to each $Z_t$ induces a definable diffeomorphism.
 We  are going to verify that  for $\delta$ sufficiently small we have for all  $x\in Z_t$ and $u \in T_xZ_t$ unit vector (for any $t$):
   \begin{equation}\label{eq_istopie_approx2}
      |x-r_t(x)|< \ep(r_t(x)) \et |u-d_x r_t(u)|<\ep(r_t(x)).
     \end{equation}
 Before proving these two inequalities, let us make it clear that this yields the desired fact. We may assume $\ep <\frac{1}{2}$.  Setting $y=r_t(x)\in M$ and $v=\frac{d_xr_t(u)}{|d_xr_t(u)|}\in T_yM$ (if $x\in Z_t$ and $u \in T_xZ_t$ is a unit vector), (\ref{eq_istopie_approx2}) immediately entails  $|d_xr_t(u)| \ge \frac{1}{2}$   so that
  $$|(r_{t})^{-1}(y)-y|< \ep(y) \et |d_y (r_{t})^{-1}(v)-v|<2\ep(y),$$
showing (\ref{eq_istopie_approx}) for $r_t^{-1}$ (up to the constant $2$).

   We can assume that $\delta<1$.
Observe that for such $x$
\begin{equation}\label{e1}|r_t(x)-x|= dist(x,M)\le |x-\varphi_t^{-1}(x)|\overset{(\ref{eq_istopie_approx})}\le  \delta(\varphi_t^{-1}(x))\end{equation}
and hence 
\begin{equation}\label{e2}|r_t(x)-\varphi_t^{-1}(x)|\le 2 \delta(\varphi_t^{-1}(x))\le 2.\end{equation}
For  $\delta'(y):=\sup\{3\delta(z)\;:\ z\in M\cap  B(y,2)\}, y\in M$, we obtain from (\ref{e1}) and (\ref{e2}) that  
\begin{equation}\label{e3}|r_t(x)-x|\le  \delta'(r_t(x)).\end{equation}

Now, as  $x\mapsto d_x r$ is uniformly continuous on bounded sets, if $\delta$ is small enough we have on $U_\delta$  for each unit vector $u\in T_{r_t(x)}M$:
\begin{equation}\label{d1}
 |d_x r(u)-u|=|d_x r(u)-d_{r( x)} r(u)| \le \ep(r_t(x)).
\end{equation}
 This is almost the desired estimate (together with (\ref{e3})). The problem is that we need to have such an estimate for $u \in T_xZ_t$, $x\in Z_t\subset U_{\delta}$. We thus are going to estimate the distance between $T_xZ_t$ and $T_{r_t(x)} M$ (see (\ref{a1}) and (\ref{a2}) below).

Denote by $\G^m_n$, $m=\dim M$, the Grassmannian of $m$-dimensional linear subspaces of  $\R^n$, that we endow with the metric $$\rho(P,Q):=\sup_{a\in P,|a|=1}\inf_{b\in Q} |a-b|.$$
Observe that 
\begin{equation}\label{a1}
\rho(T_xZ_t,T_{\varphi^{-1}_t(x)}M)\overset{(\ref{eq_istopie_approx})}<\delta(\varphi_t^{-1}(x))<\delta'(r_t(x)),
\end{equation}
by definition of $\delta'$ and (\ref{e2}). Moreover, since the tangent bundle of $M$ is at least $\Cc^1$, by (\ref{e2}),  for $\delta$ sufficiently small, we have \begin{equation}\label{a2}
\rho(T_{r_t(x)}M,T_{\varphi^{-1}_t(x)}M)<\ep(r_t(x)).
\end{equation}
By (\ref{a1}) and (\ref{a2}), we get for $\delta$ small enough
$$\rho(T_xZ_t,T_{r_t(x)}M)<2\ep(r_t(x)),$$ 
 which with (\ref{e1}) and (\ref{d1}) yields (\ref{eq_istopie_approx2}).
\end{proof}

\subsection*{Acknowledgements}
Research supported by the National Science Center (Poland) under  grant number  2021/43/B/ST1/02359.


\end{document}